\documentclass{amsart}
\usepackage{graphicx,calc,amscd}
\usepackage{float}
\usepackage{amsfonts}
\usepackage{amssymb}
\usepackage{fontenc}
\usepackage{amsmath}
\usepackage{epsfig}
\usepackage[all]{xy}
\usepackage[usenames]{color}

\def\R{{\mathbb R}}
\def\C{{\mathbb C}}
\def\Z{{\mathbb Z}}

\def\T{{\mathbb T}}

\def\al{{\alpha}}
\def\ga{{\gamma}}

\def\Si{{\Sigma}}

\def\Te{{\Theta}}
\def\e{{\varepsilon}}

\newcommand{\Sys}{\mathfrak{S}}
\newcommand{\vol}{{\rm vol}}

\def\a{{\bf a}}
\def\b{{\bf b}}

\def\G{{\mathcal G}}
\def\L{{\mathcal L}}

\def\H{{\mathcal H}}

\def\F{{\mathcal F}}
\def\inf{\mbox{inf}}
\def\vol{\mbox{vol}}

\def\sys{\mbox{sys}}

\def\*sys{\mbox{*sys}}
\def\*'sys{\mbox{*'sys}}

\def\log{\mbox{log}}

\numberwithin{equation}{section}

\newtheorem{innertheorem}{Theorem}

\newenvironment{maintheorem}[1]
    {\renewcommand\theinnertheorem{#1}\innertheorem}
    {\endinnertheorem}

\newtheorem{innerproposition}{Proposition}

\newenvironment{mainproposition}[1]
    {\renewcommand\theinnerproposition{#1}\innerproposition}
    {\endinnerproposition}

\newtheorem{theorem}{Theorem}[section]
\newtheorem{proposition}[theorem]{Proposition}
\newtheorem{corollary}[theorem]{Corollary}
\newtheorem{lemma}[theorem]{Lemma}
\newtheorem{definition}[theorem]{Definition}
\newtheorem{conjecture}[theorem]{Conjecture}
\newtheorem*{definition*}{Definition}
\newtheorem*{proposition*}{Proposition}
\newtheorem*{corollary*}{Corollary}
\newtheorem*{conjecture*}{Conjecture}

\theoremstyle{definition}
\newtheorem{exemple}[theorem]{Example}
\newtheorem{remark}[theorem]{Remark}
\newtheorem{question}[theorem]{Question}

\title[Systolic volume of homology classes]{Systolic volume of homology classes}

\author[I.~Babenko]{Ivan Babenko}
\address{Institut de Math\'ematiques et de Mod\'elisation de Montpellier - Universit\'e \linebreak  Montpellier 2}
\email{babenko@math.univ-montp2.fr}

\author[F.~Balacheff]{Florent Balacheff}
\address{Laboratoire Paul Painlev\'e - Universit\'e Lille 1}
\email{Florent.Balacheff@math.univ-lille1.fr}

\date{\today}



\thanks{This work was partially supported by the following grants: RFSF 10-01-00257-a and ANR 12-BS01-0009.}

\begin{document}

\begin{abstract}
Given an integer homology class of a finitely presentable group,
the systolic volume quantifies how tight could be a geometric
realization of this class.
%Given a pair $(G,\a)$ where $G$ is a finitely presentable group and $\a$ an integer homology class of $G$, the systolic volume is defined as the least volume of any geometric realisation of the homology class by a pseudomanifold with a polyhedral metric for which the length of curves representing elements of $G$ are at least $1$.
In this paper, we study various aspects of this numerical
invariant showing that it is a complex and powerful tool to
investigate topological properties of homology classes of finitely
presentable groups.
%Among other results, we show that the systolic volume of the pair $(G,k\a)$ is bounded from above by the function $C(G,\a) \cdot {k\over \ln (1+k)}$ where $C(G,\a)$ is some constant depending on $(G,\a)$. We also bound from below the systolic volume of a pair by the torsion of the group.
\end{abstract}

\maketitle

\section{Introduction}

Let  $G$ be a finitely presentable group and fix
a non trivial homology class $\a \in H_m(G,\Z)$ of dimension $m\geq 1$.
In order to study topological and algebraic properties of the class $\a$,
we consider the various ways this class can be realized by
a pseudomanifold endowed with a polyhedral metric. For such realizations,
the two main geometrical quantities are the volume of the pseudomanifold and
the length of loops representing non-trivial elements of $G$.
The systolic volume turns out to be the simplest way to compare
these geometric quantities in order to form a topological invariant and is defined as follows.\\

A {\it geometric cycle $(X,f)$ representing a class $\a$} is a
pair $(X,f)$ consisting of an orientable pseudomanifold $X$ of
dimension $m$ and a continuous map $f : X \to K(G,1)$ such that
$f_\ast[X]=\a$ where $[X]$ denotes the fundamental
class of $X$ and $K(G,1)$ the Eilenberg-MacLane
space\footnote{This definition slightly differs from the original
one in \cite{Gro83} where in addition a geometric cycle is
provided with a polyhedral metric on $X$.}. The representation is
said to be {\it normal} if in addition the induced map $f_\sharp :
\pi_1(X) \to G$ is an epimorphism. Given a geometric cycle
$(X,f)$ and a polyhedral metric $g$ on $X$ (see \cite{Bab06} for a definition), the  {\it relative homotopic systole} $\sys_f(X,g)$ is defined as the least length of a loop $\gamma$
of $X$ such that $f\circ \gamma$ is not contractible. The {\it
systolic volume} of $(X,f)$ is then the value
$$
\Sys_f(X):=\underset{g}{\inf}\, \frac{\vol(X,g)}{\sys_f(X,g)^m},
$$
where the infimum is taken over all polyhedral metrics $g$  on $X$
and $\vol(X,g)$ denotes the $m$-dimensional volume of $X$. When $f : X \to K(\pi_1(X),1)$ is the classifying map
(induced by an isomorphism between fundamental groups), we
simply denote by $\Sys(X)$ the systolic volume of the pair
$(X,f)$.  From \cite[Section 6]{Gro83}, we have that for any  dimension $m\geq 1$
$$
\sigma_m:= \underset{(X,f)}{\inf}\, \Sys_f(X) >0,
$$
the infimum being taken over all geometric cycles $(X,f)$
representing a non trivial homology class of dimension  $m$. The
following notion was introduced by Gromov (see \cite[Section
6]{Gro83}):

\begin{definition*}\label{def:sys}
The systolic volume of the pair $(G, \a)$ is defined as the number
$$
\Sys(G, \a):=\underset{(X,f)}{\inf} \, \Sys_f(X),
$$
where the infimum is taken over all geometric cycles $(X,f)$
representing the class $\a$.
\end{definition*}

Any integer class admits a representation by a geometric cycle, see
Theorem \ref{th:realisationintro} below. 
The systolic volume of $(G,\a)$ is thus well defined and
satisfies $\Sys(G, \a)\geq \sigma_m$. But it is not clear whether the
infimum value $\Sys(G, \a)$ is actually a minimum and what the
structure of a minimal geometric cycle would be. When the class  $\a$ is representable by a manifold, the systolic volume coincides with the systolic volume of any normal representation of $\a$ by a manifold, see
\cite{Bab06, Bab08, Bru07}. A manifold is an example of {\it
admissible} pseudomanifold, that is a special
type of pseudomanifolds for which any element of the fundamental
group can be represented by a curve not going through the singular
locus of $X$. In this article, we first prove the following
result, see section \ref{sec:representation}.

\begin{maintheorem}{I} \label{th:realisationintro}
Let $G$ be a finitely presentable group.
Any homology class $\a \in H_m(G, \Z)$ admits a normal representation by an admissible geometric cycle.
Furthermore, any such representation by an admissible geometric cycle  $(X,f)$ is minimal in the sense that
$$
\Sys_{f}(X)=\Sys(G, \a).
 $$
\end{maintheorem}

Here an admissible geometric cycle $(X,f)$ stands for a geometric
cycle whose pseudomanifold $X$ is admissible. 
Thus the infimum in the definition of systolic volume of a homology class is a minimum.
Furthermore the systolic volume of an admissible orientable pseudomanifold $X$ depends only
on the image of its fundamental class $f_\ast[X] \in H_m(\pi_1(X),\Z)$.
In section \ref{sec:nilp} we exhibit an example showing that the condition
of normalization (that is, $f_\sharp$ is an epimorphism between fundamental groups)
can not be relaxed in our theorem. Finally observe that working with the class of pseudomanifolds is necessary and not merely a formal generalization: there exist classes which are not representable by manifolds, see \cite{Thom54}.\\

In order to understand the systolic volume invariant, we study its
distribution along the real line. Here two phenomena appear.
First, the systolic volume function does not avoid arbitrarily
large intervals.

\begin{mainproposition}{I}\label{prop:densintro}
Let $m \geq 3$. For any interval $I \subset \R^+$ of length at
least $\sigma_m$, there exists a pair $(G,\a)$ consisting of a
finitely presentable group and a homology class of dimension $m$
such that $\Sys(G,\a) \in I$.
\end{mainproposition}

Secondly, there is no finiteness result for systolic volume in
dimension $m\geq 3$. In order to avoid irrelevant non-finiteness results, we
introduce the following definition.  A class $\a \in H_m(G,\Z)$ is
said to be {\it reducible} if there exists a subgroup $H \varsubsetneq
G$ and a class $\b \in H_m(H, \Z)$ such that $i_{\ast}(\b) = \a$
where $i$ denotes the canonical inclusion. Otherwise the class
will be called {\it irreducible}. If the class is reducible then
$\Sys(G,\a)\leq \Sys(H,\b)$ by definition. Now observe that it is trivial to form infinite sequences of classes with bounded systolic volume which reduce to the same class. Indeed consider $b \in H_m(H,\Z)$ a homology class of dimension $m\geq 2$. Then for any positive integer $n$
$$
\Sys(H\ast\underset{n}{\underbrace{\Z \ast \ldots
\ast\Z}},\a_n)\leq \Sys(H,\b)
$$
where $\a_n=(i_n)_\ast(\b)$ denote the image of the class $\b$ by the natural inclusion $i_n : H \hookrightarrow H\ast\Z \ast \ldots
\ast\Z$.
The existence of infinite sequences of pairewise distinct irreducible classes with bounded systolic volume is much more meaningfull and will be proven in section \ref{sec:finitude}.

\begin{maintheorem}{II}\label{th:nonfinitudeintro}
For any dimension $m\geq 3$ there exists an infinite sequence of
finitely presentable pairwise distinct groups
$\{G_i\}$ for which at least one irreducible class  $\a_i \in
H_m(G_i, \Z)$ satisfies $\Sys(G_i,\a_i) \leq 1$.
\end{maintheorem}

So we have to introduce topological or algebraic restrictions in
order to hope finiteness results. For instance, given a finitely
presentable group $G$, a homology class $\a \in H_m(G,\Z)$ and a positive number $T$, the
number of integer multiple classes $k\a$ whose systolic volume is less than $T$ is at least $T . \ln T$ (up to some multiplicative constant). More precisely, the following theorem will be proven in section
\ref{sec:multiple}.

\begin{maintheorem}{III}\label{th:multipleintro}
Let $G$ be a finitely presentable group and $\a \in H_m(G, \Z)$
where $m\geq 3$. There exists a positive number $C(G,\a)$
depending only on the pair $(G,\a)$ such that
$$
\Sys(G,k\a) \leq C(G,\a) \cdot \frac{k}{\ln(1+k)}
$$
for any positive integer $k$. In particular,
$$
\lim_{k \to \infty} \frac{\Sys(G,k\a)}{k}=0.
$$
\end{maintheorem}

This result shows that the
systolic volume of multiples of a class is a sublinear function, which is remarkable.
For some special classes $\a$ (for which simplicial volume is not
zero, see section \ref{sub:hyp}), we know after Gromov
\cite{Gro83} that there exists a positive number $C'(G,\a)$
depending only on the pair $(G,\a)$ such that
\begin{equation}\label{eq:lower.bound}
\Sys(G,k\a) \geq C'(G,\a) \cdot \frac{k}{(\ln(1+k))^m}.
\end{equation}
Moreover, for fundamental groups $\pi_l$ of orientable surfaces
$\Sigma_l$ of genus $l\geq 1$ and for the corresponding
fundamental classes $[\Sigma_l]$, we know by \cite{Gro83} and
\cite{BS94} that
$$
\Sys(\pi_l,k[\Sigma_l])\sim \frac{k}{(\ln(1+k))^2},
$$
where $f\sim g$ means that there exists some positive constants
$c$ and $C$ such that $c.f \leq g \leq C.f$. This naturally leads
to the following conjecture.

\begin{conjecture*}
Let $G$ be a finitely presentable group and $\a \in H_m(G, \Z)$ a class of non zero simplicial volume where
$m\geq 3$. Then 
$$
\Sys(G,k\a)\sim \frac{k}{(\ln(1+k))^m}.
$$
\end{conjecture*}

In section \ref{sec:nilp}, we explore the case of nilmanifolds,
and more specifically the case of the Heisenberg group of
dimension $3$. We obtain a new illustration of the possible
behaviour of the systolic volume of cyclic coverings. The study of
the systolic volume of cyclic coverings in terms of the number of
sheets has been suggested in \cite{Gro96}, and the first result in
this direction can be found in \cite{BB05}.

The Heisenberg group $\H$ of dimension $3$ is the group of
triangular matrices
$$
\left\{\left(
\begin{array}{ccc}
1 & x & z \\
0 & 1 & y \\
0 & 0 & 1
\end{array}
\right) \mid \, x, y, z \in \R \right\}.
$$
The subset $\H(\Z)$ of $\H$ composed of matrices with integer
coefficients ({\it i.e.} matrices for which  $x, y, z \in \Z$) is
a lattice, and we will denote by $M_{\H} = \H / \H(\Z)$ the
corresponding nilmanifold. First of all, we obtain the
following explicit upper bound for the systolic volume of
multiples of the fundamental class $M_\H$.
\begin{maintheorem}{IV}\label{th:Heisintro}
Let $\a = [M_{\H}] \in H_3(\H(\Z), \Z)$ be the fundamental class
of $M_\H$. Then
$$
\Sys(\H(\Z),k\a) \leq 19 \cdot \Sys(\H(\Z),\a)
$$
 for any integer  $k\geq 1$.
\end{maintheorem}
The constant appearing here is the one involved in the resolution
of the classical Waring problem (see \cite{BDD86}): any integer
number decomposes into a sum of at most $19$ fourth powers. The
idea of using the solution of the Waring problem in order to bound
from above the function $\Sys(G, k\a)$ when
$(G,\a)=(\H(\Z),[M_{\H}])$  carries over
 to any pair $(G,\a)$ where $G$ is a nilpotent graded group without torsion and $\a$ denotes the fundamental class of the corresponding nilmanifold, see Theorem \ref{th:nil}.

Now consider the sequence of lattices
$\{\H_n(\Z)\}_{n=1}^{\infty}$ of $\H$, where $\H_n(\Z)$ denotes
the subset of matrices whose integer coefficients satisfy $x \in
n\Z$ and $y, z \in \Z$. Denote by $M_{\H_n} = \H / \H_n(\Z)$ the
corresponding nilmanifolds. The manifold $M_{\H_n}$ is a cyclic
covering with $n$ sheets of $M_\H$, and the proof of Theorem \ref {th:multipleintro} implies that
$$
\Sys(M_{\H_n}) \leq C \cdot \frac{n}{\ln (1+n)},
$$
see Remark \ref{rem:cyclic}.
The fact that the function $\Sys(M_{\H_n})$ goes to infinity is not obvious, and its proof relies on a new invariant for groups called {\it simplicial complexity} and defined in \cite{BBB14}. More precisely, we have the following.

\begin{proposition*}\cite{BBB14}\label{cor:Heisintro}
The function $\Sys(M_{\H_n})$ satisfies the following inequality:
$$
\Sys(M_{\H_n}) \geq a\frac{\ln n}{\exp(b\sqrt{\ln \ln n})} ,
$$
where $a$ an $b$ are two positive constants. 
\end{proposition*}

In particular,
$$
\lim_{n\to +\infty} \Sys(M_{\H_n})=+\infty.
$$
Note that in this case $\| M_{\H_n} \|_{\Delta} = 0$ and the
lower bound (\ref{eq:lower.bound})
%second bound of Theorem \ref{th:gromovintro}
does not apply. For any integer $n$
the manifold $M_{\H_n}$ gives a non-normal realization of the
class $n[M_\H]$.
So the normalization's condition in Theorem \ref{th:realisationintro} cannot be relaxed.\\

\section{Manifolds with singularities as extrema of the systolic volume}\label{sec:representation}

The systolic volume of a homology class is defined by an infimum. We may wonder if this infimum is reached and what would be the structure of a minimizing pseudomanifold. If the homology class $\a$ is realized by a manifold, we know that its systolic volume is a minimum which coincides with the systolic volume of any normal representation of $\a$ by a manifold, see \cite{Bab06, Bab08, Bru07}. But there exist classes  which do not admit representations by manifolds (see \cite{Thom54}), and for such a class $\a$ it is not even clear whether $\Sys(G, \a)$ is actually a minimum.\\

Let $X$ be a pseudomanifold of dimension $m$. The {\it singular
locus} of $X$ is by definition the set $\Si(X)$ of points of $X$
which do not have a neighbourhood homeomorphic to an
$m$-dimensional ball. By definition of a pseudomanifold, $\Si(X)$
is a simplicial subcomplex of codimension at least $2$.

\begin{definition} A pseudomanifold $X$ is said to be {\it admissible}
if the natural inclusion $X \setminus \Si(X) \subset X$ induces
an epimorphism on fundamental groups.
\end{definition}

That is, a pseudomanifold is admissible if any element of the
fundamental group can be represented by a loop of $X \setminus
\Si(X)$. A geometric cycle $(X,f)$ representing some homology
class $\a$ will be called {\it admissible} if the pseudomanifold
$X$ is admissible.

\begin{exemple} Let $M$ be a triangulated manifold and $N \subset M$ be
a simplicial subcomplex  of codimension $\geq 2$. Denote by ${\sharp}N$
the set of connected components of $N$.
The simplicial complex $M /{\sharp}N$ obtained from $M$ by
contracting the connected components of $N$ into distinct
points is an admissible pseudomanifold. The singular locus $\Si(M
/{\sharp}N)$ consists of the points corresponding to the connected
components of $N$.
\end{exemple}

\begin{exemple}\label{ex:sing}
Let $M$ be a manifold with boundary $\partial M
= A \times P$ where $A$ is a manifold and $P$ is a connected
manifold. The result of the fibred contraction of $P$ is an
admissible pseudomanifold $\overline{M}$ homeomorphic to the space
$$
M \mathop{\cup}\limits_{\partial M} A\times CP
$$
where  $CP$ stands for the cone over $P$ and the singular locus
$\Si(\overline{M})$ is homeomorphic to $A$. Remark that if  $P$
is simply connected  then $\pi_1(\overline{M})=\pi_1(M)$. The
pseudomanifold $\overline{M}$ obtained that way is a particular
example of {\it singular $P$-manifold}, see \cite{Baas73} and subsection \ref{sec:modelesing}
for the general construction.
\end{exemple}

Remark that an admissible pseudomanifold of dimension $2$ is a
surface. In particular it does not possess any singularity. \\

Theorem \ref{th:realisationintro} is a direct consequence of the following two propositions.

\begin{proposition}\label{thm:realisation} Let $G$ be a finitely presentable
group and $\a \in H_m(G, \Z)$ be a homology class of dimension $m \geq 3$.
Suppose that there exists a normal representation of the class $\a$ by
an admissible geometric cycle $(X,f)$. Then
$$
\Sys(G, \a)=\Sys_{f}(X).
 $$
\end{proposition}
The condition of normalisation saying that {\it $f_\sharp$ is an epimorphism}
can not be dropped, see section \ref{sec:nilp} and the example of the Heisenberg group. \\

The following proposition, together with Proposition \ref{thm:realisation},
shows that for any pair $(G,\a)$  the systolic volume $\Sys(G,
\a)$ is actually a minimum.

\begin{proposition}\label{prop:representation} Let $K$ be a
CW-complex whose fundamental group is finitely
generated.
%whose fundamental group is finitely presentable
Any homology
class $\a \in H_*(K, \Z)$ admits a normal representation by an
admissible geometric cycle $(X,f)$.
\end{proposition}

Before proving Proposition \ref{thm:realisation} and Proposition \ref{prop:representation}, we
need some technical results.

\subsection{Technical lemmas}

Hopf's trick perfectly adapts to the setting of admissible
pseudomanifolds. Consider a map
$$
f: (X, X_1) \longrightarrow (Y, Y_1)
$$
between two relative manifolds of the same dimension $m \geq 3$.
Suppose that  $f$ is transversal at $y \in Y \setminus Y_1$, {\it
i.e.} there exists an embedded  $m$-disk $D$ such that
\begin{itemize}
\item $y \in D \subset Y \setminus Y_1$ ;
\item $f^{-1}(D) = \cup_{i=1}^n D_i $ is a disjoint union of $m$-disks embedded in $X \setminus X_1$ ;
\item the restriction of $f$ to $f^{-1}(D)$ is a covering map with base space $D$ and $n$ sheets.
\end{itemize}
Set $x_i = D_i \cap f^{-1}(y)$. The technical trick by Hopf is
essentially contained in the the following lemma.

\begin{lemma}\label{lem:Hopf1}
Suppose that there exists an  $m$-disk embedded in  $D' \subset X
\setminus X_1$ with the following properties:

\begin{enumerate}
\item $D_1,D_2 \subset D'$ and $D' \cap D_i = \emptyset$ if $i >2$ ;

\item for any path $\gamma$ from $x_1$ to  $x_2$ in  $D'$, the loop
$f(\gamma)$ is contractible in $Y$ ;

\item for any orientation of $D$, the orientations on  $D_1$ and $D_2$ induced  by $f$ are not compatible in $D'$.

\end{enumerate}
Then there exists a homotopy $\{f_t\}_{0 \leq t \leq 1}$ of $f_0 =
f$ which is constant on $X \setminus D'$ and such that
$$
f^{-1}_1(D) = \mathop{\cup}\limits_{i > 2} D_i,
$$
the last union being empty if $n=2$.
\end{lemma}

We refer the reader to \cite[p.378-380]{Eps66} for a proof of Lemma \ref{lem:Hopf1}, which will work in our setting, as the
construction of the homotopy $f_t$ occurs in $D'$  and, therefore, carries
over to our context. We will now state the corresponding version of Hopf's
Theorem in the orientable case.

\begin{lemma}\label{lem:Hopf2}
Let  $X$ be an admissible orientable connected pseudomanifold and
$(Y, Y_1)$ an orientable relative manifold of the same dimension
$m \geq 3$. Suppose that $f : X \longrightarrow Y$ is a map of
degree $k$ inducing an epimorphism between fundamental groups.

Then there exists a homotopy $\{f_t\}_{0 \leq t \leq 1}$ of $f_0 =
f$ such that $f^{-1}_1(Y \setminus Y_1)$ is homeomorphic to the
disjoint union of $k$ disks and the restriction of $f_1$ to the
union of these disks is a covering map with base space $Y
\setminus Y_1$ and $k$ sheets.
\end{lemma}

The degree of $f$ stands here for the absolute value of the
multiple defined by the induced map $f_*: H_m(X; \Z)
\longrightarrow  H_m(Y, Y_1; \Z) $. A corresponding version of
Lemma \ref{lem:Hopf2} also holds in the non-orientable context
with the notion of  absolute degree.

\begin{proof}
Consider a point $y \in Y \setminus Y_1$. We can assume that $ y
\notin f(\Sigma(X)) $ and the map $f$ to be transversal at $y$.
Let  $D \subset Y \setminus Y_1$ be a disk containing the point
$y$ such that $f^{-1}(D) = \cup_{i=1}^n D_i $ is a disjoint union
of $m$-disks embedded in $X \setminus \Sigma(X)$ and such that the
restriction of $f$ to $\cup_{i=1}^n D_i$ is a covering map with
$n$ sheets and base space  $D$. We have $n \geq k$ and suppose
that $n > k$. We can choose generators of $H_m(X; \Z)$ and $H_m(Y,
Y_1; \Z)$ such that the map $f_*$ induced on $m$-dimensional
homology is simply the multiplication by $k$. This induces an
orientation both on  $X$ and $Y \setminus Y_1$, and also on disks
$\{D_i\}_{i=1}^n $ and $D$. As $n > k$, there exists two disks say
$D_1$ and $D_2$  such that  $f|_{D_1}$ reverses the orientation
and $f|_{D_2}$ preserves it. We now follow step by step the proof
of \cite[Theorem 4.1]{Eps66}. Join the two points $x_i = f^{-1}(y)
\cap D_i, \ i = 1, 2$ by a simple curve $\ga \subset X \setminus
\Si(X)$. Because  $f$ induces an epimorphism between fundamental
groups, there exists a loop  $\al$ based at $x_1$ such that
$f(\al)$ and $f(\ga)$ are homotopic as loops based at  $y$. As $X$
is admissible, we can furthermore choose  $\al$ in $X \setminus
\Si(X)$. The  concatenation $\al^{-1}\star \ga$ and
its evident modification in a neighborhood of the
concatenation defines a simple curve $\beta \subset X\setminus
\Si(X) $ joining $x_1$ and $x_2$. The loop $f(\beta)$ is
contractible in  $Y$ relatively to $y$. We then define the disk
$C$ as a small enough neighborhood of $\beta$, and apply Lemma
\ref{lem:Hopf1}. Remark that the choice of $C$ implies a possible
diminution of the size of the disks $\{D_i\}_{i=1}^n $ and $D$.
The end of the proof is straightforward, see \cite{Eps66} for the
missing details.
\end{proof}

\subsection{Admissible geometric cycles are minima of the systolic volume}
In this subsection, we prove Proposition \ref{thm:realisation}. For
this we use in a decisive way the comparison and extension
techniques elaborated in \cite{Bab06}, as well as the ideas
contained in this article. For the reader's convenience, let first recall these two techniques.

\smallskip

We first present the comparison principle which originally appeared as Proposition 3.2 in \cite{Bab06}.
Given two pseudomanifolds $X_1$ and $X_2$ recall that a simplicial map from $X_1$ to $X_2$ is called $m$-monotone if the preimage of any open $m$-simplex is either an open $m$-simplex or empty. 

\begin{proposition}[Comparison Principle]
Let $(X_1,f_1)$ and $(X_2,f_2)$ be two geometric cycles representing an homology class $\a \in H_m(G,\Z)$ and suppose that there exists a $m$-monotone simplicial map $p : X_1 \to X_2$ such that $f_1=f_2 \circ p$. Then
$$
\Sys_{f_1}(X_1) \leq \Sys_{f_2}(X_2).
$$
\end{proposition}

Now we state the extension principle, see Proposition 3.6 in \cite{Bab06}.

\begin{proposition}[Extension Principle]
Let $(X_1,f_1)$  be a geometric cycle representing an homology class of dimension $m$. Suppose that $X_2$ is a CW-complex obtained from $X_1$ by adding a finite number of cells with dimension from $3$ to $m-1$. Then
$$
\Sys_{f_1}(X_1) = \Sys_{f_2}(X_2)
$$
where $f_2:X_2 \to K(G,1)$ is the unique (up to homotopy) extension of $f_1$.
\end{proposition}

We now present the proof of Proposition \ref{thm:realisation}. 

\smallskip

Let $(X_1,f_1)$ be a geometric cycle representing the class $\a$. The pseudomanifold $X_1$ admits a cell decomposition  with only
one  $m$-cell (see for example \cite{Sab06}). This allows us to
describe $X_1$ as a relative $m$-manifold $(Y, Y')$, where $Y$
denotes the  $m$-cell and $Y'$ lies in the $(m-1)$-skeleton.
Following \cite{Bab06} we
construct an extension of  $X_1$ as follows. 

\smallskip
We start by adding a finite number of $1$- and $2$-cells to $X_1$ such
that the resulting CW-complex $X_1(2)$ has fundamental group $G$ and the inclusion map $i : X_1 \hookrightarrow X_1(2)$ satisfies $ i_\sharp=(f_1)_\sharp$. More precisely, given a finitely generated presentation of $G$  we first construct a $2$-dimensional finite cellular complex $K^{(2)}$ by considering the wedge sum of 1-cells corresponding to the generators of the presentation of $G$, and then gluing $2$-cells according to its relations. The finite cellular complex $X(2)$ is then obtained as the wedge sum of $X_1$ and $K^{(2)}$, with additional $2$-cells glued along the image of $(f_1)_\sharp(\gamma_i)\ast \gamma_i^{-1}$ for a finite set $\{\gamma_i\}_{i \in I}$ of generators of $\pi_1(X_1)$. Observe that 
\begin{lemma}
$$
\Sys_{f_1}(X_1)=\Sys(X(2)).
$$
\end{lemma}

\begin{proof}[Proof of the Lemma]
In fact, using the comparison principle, we first observe that 
$$
\Sys_{f_1}(X_1)\leq\Sys(X_1(2)).
$$
Now fix a positive $\e$ and a metric $g_1$ on $X_1$ such that
$$
{\vol(X_1,g_1) \over \sys_{f_1}(X_1,g_1)^m}\leq \Sys_{f_1}(X_1)+\e.
$$
Now define the metric $g$ on $X_1(2)$ which coincides with $g_1$ on $X_1$, for which the length of each additional $1$-cell is exactly $\sys_{f_1}(X_1,g_1)$ and for which each additional $2$-cell is a round hemisphere whose curvature is chosen accordingly to the length of its boundary. It is straightforward to see that any closed curve of $X_1(2)$ can be homotoped to a curve lying in the union of $X_1$ and the $1$-skeleton of  $K^{(2)}$ without increasing the length. In particular $\sys(X_1(2),g)=\sys_{f_1}(X_1,g_1)$. But $\vol(X_1(2),g)=\vol(X_1,g_1)$ for dimensional reasons as $m\geq 3$ so we get 
$$
\Sys(X_1(2))\leq {\vol(X_1(2),g) \over \sys(X_1(2),g)^m}\leq \Sys_{f_1}(X_1)+\e
$$
which concludes the proof.
\end{proof}

\smallskip
 
Then, for each dimension $k$ going from $3$ to $m$, we add
$k$-cells to $X_1(k-1)$ such that the new CW-complex $X_1(k)$ thus
obtained satisfies $\pi_i(X_1(k))=0$ for $1<i<k-1$. At the end we
get a  CW-complex denoted by $X_1(m)$ which is $m$-aspherical and whose
fundamental group is $G$.  Remark that $Y \setminus
Y'$ is an $m$-cell of  $X_1(m)$.
By adding cells of dimension higher than
$m$, we can realize the Eilenberg-MacLane space $K(G,1)$ as an
extension of $X_1(m)$. 

\smallskip

By assumption there exists an admissible pseudomanifold  $X$ and
a map $f : X \to K(G,1)$ giving a realization of the class
$\a \in H_m(G, \Z)$ such that  $f_\sharp : \pi_1(X) \to G$ is an
epimorphism. By the cellular approximation theorem and according to Lemma 3.10 of \cite{Bab06} which applies to this situation, we can find a map 
$$
g : X \to X_1(m)\subset K(G,1)
$$
homotopic to $f$ in $K(G,1)$ such that
\begin{equation}\label{eq:coherence}
g_*[X] = i_*[X_1] 
\end{equation}
in $H_m(X_1(m), \Z)$. As $g$ is homotopic to $f$, the induced map $g_\sharp$ is still an epimorphism and  
$$
\Sys_f(X)=\Sys_g(X).
$$
  Let $\{Y\}\cup \{Y_i\}_{i\in I}$ denote the
$m$-cells of $X_1(m)$ (this list can be finite or infinite). To
each $m$-cell $Y_i$ or $Y$ is associated the relative manifold
$(Y_i, \widehat{Y_i})$ or $(Y, \widehat{Y})$, where
$\widehat{Y_*}$ denotes the closure of the union of all the other
cells of $Y_*$. The map $g$ induces maps
$$
\widetilde{g}: X \to (Y, \widehat{Y}) \ \ \mbox{and} \ \
\widetilde{g_i}: X \to (Y_i, \widehat{Y_i}), \ \ \forall i \in I
$$
which still induce epimorphisms between fundamental groups.
From (\ref{eq:coherence}), we deduce that the degree of
$\widetilde{g}$ is equal to  $1$, and that the degree of each
$\widetilde{g_i}$ is zero. As $X$ is compact, $g(X)$ intersects
only a finite number of  $m$-cells. We then apply  Lemma
\ref{lem:Hopf2} to each relative manifold $(Y_i, \widehat{Y_i})$
such that  $Y_i \bigcap g(X) \neq \emptyset$. In this way we
obtain a map
$$
g_1 : X \to Z \subset X_1(m-1)
$$
homotopic to $g$ where  $Z$ denotes the subcomplex of $X_1(m-1)$
obtained from $X_1(2)$ by adding the cells with
dimension from $3$ to $m-1$ which intersect $g_1(X)$. Observe that these cells are in finite number so 
$$
\Sys(Z)=\Sys(X_1(2))
$$
according to the extension principle. The map can be chosen to be $m$-monotone and simplicial, see  \cite[Proposition 3.13]{Bab06}. Thus
$$
\Sys_g(X)=\Sys_{g_1}(X)\leq \Sys(Z)
$$ 
according to the comparison principle.

\smallskip

To summarize, we have shown that
$$
\Sys_{f}(X) \leq \Sys_{f_1}(X_1)$$
for any representation $(X_1, f_1)$ of $\a$ which proves that
$$
\Sys_{f}(X) = \Sys(G,\a).
$$

\subsection{Manifolds with singularities of prescribed type according to Baas}\label{sec:modelesing}
%of prescribed singularity type

In order to prove Proposition \ref{prop:representation}, we briefly
recall a type of singular manifolds introduced by N. Baas. 
We follow the presentation of \cite{Baas73}.

\smallskip

First recall the following definition due to J. Cerf, see \cite{Cerf61}.
\begin{definition}
A manifold with general corners of dimension $m$ is a Hausdorff space locally modeled on the $m$-cube $[0,1]^m$.
\end{definition}

In other words, manifolds with general corners are manifolds whose boundary looks like the boundary of  the cube. A manifold with general corners can be smoothed through a process which has been well studied   (see \cite{Cerf61} for instance).

\smallskip

 To define manifolds with reasonable singularities, N. Baas---following D. Sullivan's idea---considers manifolds with general corners whose boundary decomposes as products and contract some terms of these products, like in Example \ref{ex:sing}. To describe his construction let recall the following definition.

\begin{definition}\label{def:decompose} A manifold  $M^m$ with general corners is said to be
  {\it decomposed} if its boundary decomposes into an union 
  $$
\partial M = \partial_0M \cup \partial_1 M \cup ...\cup \partial_n M
$$
of codimension $1$ submanifolds with general corners. Here the union means the identification along a common part of the boundaries of the $\partial_i M$'s.
\end{definition}

If $M$ is a decomposed manifold, each of its boundary terms turns out to be a decomposed manifold too. More precisely, by setting
\begin{equation}\label{eq:decompose}
\left\{\begin{matrix}
\partial_j(\partial_iM) & = & \partial_jM \cap \partial_iM & \mbox{if} \ \ j \neq i, \\
\partial_i(\partial_iM) & = & \emptyset  &  \\
\end{matrix}
\right.
\end{equation}
we get
$$
\partial(\partial_iM) = \mathop{\cup}\limits_{j=0}^n \partial_j(\partial_iM)
$$
and each  $\partial_iM$ is again a decomposed manifold.

\begin{exemple} If  $M$ denotes the $m$-dimensional cube, its boundary is naturally decomposed into $(m-1)$-faces :
$$
\partial M = \partial_0M \cup \partial_1M \cup ...\cup \partial_{2m-1}M.
$$
\end{exemple}

We consider a finite sequence of closed manifolds $S = \{P_0=\ast, P_1,
..., P_n \}$ ordered by increasing dimension. These manifolds will play the role of $_{``}$singularities$_{"}$ after their contraction. Before the contraction process, the following definition describes how the manifolds should look like in order to prescribe the type of singularities we will obtain.

\begin{definition}\label{def:modelesingular}
A manifold $M$ with general corners is said to be an {\it $S$-manifold} if

\noindent 1) For any subset $\omega \subset \{0, 1, ..., n\}$,
there exists a decomposed manifold $M(\omega)$ such that

\noindent (a) $M(\emptyset) = M$,

\noindent (b) $\partial M(\omega) =  \cup_{i \notin \omega}
\partial_i M(\omega)$.

\noindent (c) For every $i \in \{1,\ldots,n\} \setminus \omega$,
there exists a diffeomorphism
\begin{equation}\label{eq:produit}
\beta(\omega, i): \partial_i M(\omega)  \simeq  M(\omega, i) \times P_i , \\
\end{equation}
where $(\omega, i) = \omega \cup \{i\}$,

\noindent 2) For any subset $\omega \subset \{0, 1, ..., n\}$, and
for all $i,j \in \{1,\ldots,n\} \setminus \omega$, the following
diagram is commutative:
\begin{equation*}
\begin{matrix}
\partial_j\partial_iM(\omega)&\stackrel{\beta(\omega, i)}{\longrightarrow}&\partial_jM(\omega, i)\times P_i&
\mathop{\longrightarrow}\limits^{\beta(\omega, i, j)\times id} & M(\omega, i, j)\times P_j\times P_i \\
|| & & & & \\
\partial_jM(\omega)\cap\partial_iM(\omega) & & & & \Bigg{\downarrow} id\times T \\
|| & & & & \\
\partial_i\partial_jM(\omega)&\mathop{\longrightarrow}\limits^{\beta(\omega, j)}&\partial_jM(\omega, i)\times P_j&
\mathop{\longrightarrow}\limits^{\beta(\omega, j, i)\times id} &
M(\omega, i, j)\times P_i\times P_j
\end{matrix}
\end{equation*}
where $T$ denotes the transposition.
\end{definition}

The first part of the definition describes the local structure of
product on the boundary of the decomposed manifold $M$. The
diagram describes how the boundary components are glued together.
We now define the particular class of singular manifolds pointed by N. Baas.

\begin{definition}\label{def:singular}
To any $S$-manifold $M$
we associate the singular manifold 
$M_S$ defined as the quotient space $M(\emptyset)/ \sim$ where
$$
a \sim b
$$
if
$$
a, b \in \partial_{i_1}...\partial_{i_k}M, \ \ i_j \geq 1, \ \ k
\geq 1 ,
$$
and
\begin{equation}\label{eq:equivalence}
\begin{matrix}
\mbox{pr}\circ\beta(i_1,... , i_k)\circ....\circ\beta(i_1, i_2)\circ\beta(\emptyset, i_1)(a) = \\
\mbox{pr}\circ\beta(i_1,... , i_k)\circ....\circ\beta(i_1,
i_2)\circ\beta(\emptyset, i_1)(b).
\end{matrix}
\end{equation}
Here
\begin{equation}\label{eq:project}
\mbox{pr} : M(i_1,..., i_k)\times P_{i_1}\times... \times P_{i_k}
\longrightarrow M(i_1,..., i_k)
\end{equation}
denotes projection on the first factor.
\end{definition}

%Si $S = \{P_0, P_1,\ldots \}$ est une suite infinie de vari\'et\'es ferm\'ees ordonn\'ee par ordre croissant de dimension telle que $P_0 = *$ et les $P_k$ avec $k \geq 1$ soient de dimension positive, toute vari\'et\'e singuli\`ere de type $S_n$ o\`u $S_n \subset S$ est un sous-ensemble fini sera dite $S$-{\it vari\'et\'e singuli\`ere}.\\

If the elements of $S$ are connected manifolds, then every
manifold $M$ with singularities of type $S$ is an
admissible pseudomanifold. If not, the following remark will be of
fundamental importance in the next section.

\begin{remark}\label{rem:connexe}  For each  $i =1,\ldots,n$, we decompose the manifold $P_i$ into connected components $Q_{ij}$ and set
$$
T = \{Q_{ij} \mid 1\leq i \leq n, 1 \leq j \leq k_i\}.
$$
Given a manifold $M_S$ with singularities of type
$S$ modeled on $M$, the local $S$-structure (\ref{eq:produit}) on
$\partial M$ defines a  local $T$-structure. The commutative
diagram of Definition \ref{def:modelesingular} and the equivalence
(\ref{eq:equivalence}) allow us to define an equivalence relation
on  $M$ denoted by $\sim_T$ and such that the projections
(\ref{eq:project}) only occur along the factors of type $Q_{ij}$.
This gives rise to a $T$-singular manifold $M_T$ defined as the
quotient $M(\emptyset)/ \sim_T$. A class for the  relation
$\sim_T$ being a subclass of the relation $\sim_S$, we get a
canonical map of degree $1$
\begin{equation}\label{def:applcanon}
q: M_T \to M_S.
\end{equation}
\end{remark}

\subsection{Realization of homology classes by admissible geometric cycles.} We now prove Proposition \ref{prop:representation}. Following Milnor \cite{Miln60} and Novikov \cite{Nov62}, the  complex cobordism ring $\Omega_*^U$ is isomorphic to the ring of integer polynomials $\Omega_*^U = \Z[x_1, x_2, ...]$ where each generator $x_k$ is of degree $2k$ and can be represented by a manifold $P_k$. Each representant $P_k$ can be chosen as a complex algebraic manifold, see  \cite{St68} for instance. But the connectivity of this complex manifold is not clear in general (if $k=p^s-1$ where $p$ is a prime number, $P_k$ can be choosen as ${\C}P^k$).
Define the following sequence of singularities
\begin{equation}\label{eq:Pk}
S = \{P_1, P_2, ...\}.
\end{equation}
According to Baas' Theorem \cite[Corollaire 5.1]{Baas73}, given a homology class $\a \in H_m(X, \Z)$, there exists  a manifold $M_S$ with singularities of type $S$ of
dimension $m$ and a map $f: M_S \to X$ such that $f_*[M_S] = \a$.
The elements in $S$ are not necessarily connected manifolds. So we
proceed as in remark \ref{rem:connexe}, and obtain in this way a
new manifold with singularities $M'$ representing
$\a$ which is now an admissible pseudomanifold. Finally, we add if
necessarily $1$-handles to $M'$ and extend the map $f'=f \circ q$
(where $q$ denotes the canonical map from $M'$ to $M_S$ of degree
$1$) in such a way that $f'_\sharp$ becomes an epimorphism between
fundamental groups. This concludes the proof.

\begin{remark}\label{rem:cobord} The admissible pseudomanifold $M'$
which realizes $\a$ can be chosen as a manifold
with singularities whose singularities are more specific. In
${\C}P^m \times {\C}P^n$ with $m \leq n$, we consider the
hypersurface of degree $(1,1)$
$$
H_{m,n} = \{z_0w_0+z_1w_1+ ... + z_mw_m = 0\},
$$
where $(z_0, z_1,..., z_m)$ and $(w_0, w_1, ..., w_n)$ denote the
homogeneous coordinates in ${\C}P^m$ and ${\C}P^n$ respectively.
The manifolds $H_{m,n}$ are known as  {\it Milnor's manifolds}.
The cobordism  classes of the $\{H_{m,n}\}_{m\leq n}$ together
with the familly  of classes $\{{\C}P^ s\}_{s \geq 1}$ give rise
to a spanning familly of $\Omega_*^U$ (see \cite{Hirz58} and
\cite{Nov62}).  The classes in  $\Omega_*^U$ are thus linear
combinations (with integer coefficients) of cobordim classes
$\{H_{m,n}\}_{m\leq n}$ and  $\{{\C}P^ s\}_{s \geq 1}$. So we can
choose the $\{P_k\}_{k \geq 1}$ as a disjoint union of some of
these manifolds endowed with an adequate orientation. Taking into
account remark \ref{rem:connexe}, an admissible pseudomanifold
which represents the  class $\a \in H_*(K, \Z)$ can be choosen as
a manifold with singularities of type $\{{\C}P^
s, H_{m,n}\}$.
\end{remark}

\section{Relative density of the values of systolic volume}\label{sec:densite}

The aim of this section is to show that the set of values of
systolic volume over the set of homology classes (resp. over the
set of orientable manifolds) of fixed dimension is a relatively
dense set in the following sense.

\begin{definition}
Given a subset $A \subset \R^+$ and a positive number $d$, $A$
is said to be {\it $d$-dense} in $\R^+$ if for any interval $I \subset
\R^+$ of length $|I|>d$, the intersection $I \cap A$ is not empty.
\end{definition}

For a fixed  dimension  $m\geq 3$, define
$$
\Sigma_m:=\{\Sys(G,\a) \mid G \, \text{finitely presentable group
and} \, \a \neq 0 \in H_m(G,\Z)\},
$$
and
$$
\sigma_m :=\underset{(G,\a) \in \Sigma_m}{\inf} \Sys(G,\a).
$$
Similarly, set
$$
\Omega_m^+:=\{\Sys(M) \mid M \, \text{orientable essential
manifold of dimension $m$}\},
$$
and
$$
\omega_m^+ := \underset{M \in \Omega_m^+}{\inf} \Sys(M).
$$
Recall that $\sigma_m>0$ by \cite{Gro83} and that an orientable
manifold $M$ is said {\it essential} if the image of its
fundamental class under its classifying map is not zero. In
particular, $\omega_m^+\geq \sigma_m$. The main result of this
section is the following proposition.

\begin{proposition}\label{th:densiterelative}
For any dimension $m \geq 3$, the set $\Sigma_m$ (resp.
$\Omega_m^+$) is $\sigma_m$-dense (resp. $\omega_m^+$-dense) in
$\R^+$.
\end{proposition}

Proposition \ref{prop:densintro} of the introduction is a direct consequence of this statement.
In order to prove Proposition \ref{th:densiterelative}, we will study the behaviour of
systolic volume under the operation of connected sum. Part of
these results will also be useful in the next section.

\begin{remark} Remark that the set $\Omega_m$ of values taken by the systolic volume over all manifolds of the same dimension $m$ (not necessarily orientable) contains the subset $\Omega_m^+$ : it is also a relatively dense set of  $\R^+$ with density $\omega_m^+$. It is not clear if this density can be decreased and the answer may depend on the parity of the dimension.
\end{remark}

Fix a morphism of groups $\pi : G \to G'$ and a
homology class $\a \in H_m(G,\Z)$.

\begin{proposition}\label{prop:modif1}
If $m\geq 3$,
$$
\Sys(G',\pi_\ast \a)\leq \Sys(G,\a).
$$
\end{proposition}

\begin{proof} According to Proposition \ref{prop:representation}, fix an admissible geometric cycle $(X, f)$ representing normally the class $\a$. The admissible geometric cycle $(X, \pi \circ f)$ is a normal representation of the class $\pi_{\ast}(\a)$, so
$$
\Sys(G', \pi_{\ast}(\a))= \Sys_{\pi \circ f}(X) \leq  \Sys_f(X) =
\Sys(G, \a)
$$
by Proposition \ref{thm:realisation}.
\end{proof}

\begin{corollary}\label{cor:somme}
Let $X_1$ and $X_2$ be two orientable admissible pseudomanifolds
of dimension $m \geq 3$. Then
$$
\max\{\Sys(X_1), \Sys(X_2)\} \leq \Sys(X_1 \# X_2),
$$
where $X_1\#X_2$ denotes the connected sum of  $X_1$ and $X_2$.
\end{corollary}

\begin{proof}
Denote by  $f_j : X_j \to K(\pi_1(X_j),1)$ the classifying map for
$j=1,2$. Observe that $\pi_1(X_1\#X_2)=\pi_1(X_1) \ast
\pi_1(X_2)$. We have a natural monomorphism $i_j :
\pi_1(X_j)\hookrightarrow \pi_1(X_1) \ast \pi_1(X_2)$ and a
natural epimorphism $s_j : \pi_1(X_1) \ast \pi_1(X_2) \to
\pi_1(X_j)$ such that $s_j \circ i_j=id_{\pi_1(X_j)}$, $s_2 \circ
i_1=0$ and $s_1 \circ i_2=0$. By Proposition \ref{prop:modif1},
\begin{eqnarray}
\nonumber
\Sys(X_1\#X_2)&=&\Sys(\pi_1(X_1) \ast \pi_1(X_2),(i_1\circ f_1)_\ast[X_1]+(i_2\circ f_2)_\ast[X_2])\\
\nonumber
&\geq&\Sys(s_j(\pi_1(X_1) \ast \pi_1(X_2)),(s_j)_\ast\circ ((i_1\circ f_1)_\ast[X_1]+(i_2\circ f_2)_\ast[X_2])) \\
\nonumber
&\geq&\Sys(\pi_1(X_j),(f_j)_\ast[X_j])\}=\Sys(X_j)\\
\nonumber
\end{eqnarray}
for $j=1,2$.
\end{proof}

Furthermore, we have the following comparison result:

\begin{proposition}\label{prop:sommeconnexe1}
Let $X_1$ and $X_2$ be two orientable pseudomanifolds of
dimension $m \geq 3$. Then
$$
\Sys(X_1 \# X_2)\leq \Sys(X_1)+ \Sys(X_2).
$$
\end{proposition}

\begin{proof} The contraction of the gluing sphere into a point gives rise to a natural projection  map
$$
p : X_1 \# X_2 \to X_1 \vee X_2
$$
which induces an isomorphism between fundamental groups if $m \geq
3$. Applying the comparison principle (see \cite[Proposition 3.2]{Bab06}), we get
$$
\Sys(X_1 \# X_2) \leq \Sys(X_1 \vee  X_2) = \Sys(X_1)+\Sys(X_2).
$$
\end{proof}

With this two comparison results, we can now prove Proposition \ref{th:densiterelative}. \\

We first prove the  $\sigma_m$-density of the set $\Sigma_m$.
Consider any interval $I \subset \R^+$ such that  $|I| >
\sigma_m$. Choose a pair $(G,\a)$ composed of a finitely
presentable group and an $m$-class of its homology such that
$\sigma_m \leq \Sys(G,\a) < |I|$. According to Proposition
\ref{prop:representation}, we can choose an admissible geometric
cycle  $(X,f)$ representing normally the class $\a \in H_m(G,\Z)$
which satisfies $\Sys_f(X)=\Sys(G,\a)$ by Proposition
\ref{thm:realisation}. If the map $f_\#$ is not a monomorphism, we
can directly contract some loops in $X$ and get a new admissible
geometric cycle $(X',f')$ representing $\a$ such that $f'_\#$ is
an  isomorphism. So
$$
\Sys_f(X)=\Sys(X').
$$
The sequence $\{a_n = \Sys(\#_n X')\}_{n = 1}^{\infty}$ is
increasing by Corollary \ref{cor:somme} and satisfies $a_{n+1} -
a_n \leq \Sys(X') < |I|$  by Proposition \ref{prop:sommeconnexe1}.
By \cite{Sab07}, we have
$$
\Sys(\#_n X) \geq C_m \frac{n}{\exp(C'_m \sqrt{\log \, n})}
$$
where  $C_m$ and $C'_m$ are two positive numbers depending only
on the dimension $m$ (strictly speaking, this inequality is proved
for manifolds, but carries over perfectly to pseudomanifolds).
So the sequence $\{a_n\}$ is not bounded and does intersect $I$. \\

The proof of the  $\omega_m^+$-density of  $\Omega_m^+$ is
similar: for any interval $I \subset \R^+$ with $|I| >
\omega_m^+$, we can argue as previously starting with the sequence
$\{a_n = \Sys(\#_n M)\}_{n = 1}^{\infty}$ where $M$ is an
orientable essential manifold of dimension $m$ with  $\omega_m^+
\leq \Sys(M) < |I|$.

\section{Non-finiteness of irreducible homology classes}\label{sec:finitude}

This section deals with the following natural question:

\begin{question} Given a positive number $T$, how many are there finitely presentable groups $G$, such that any essential (orientable) manifold $M$ of dimension $m$ with fundamental group $G$ satisfies $\Sys(M) \leq T$ ?
\end{question}

In dimension $2$, this number is bounded from below by $c \cdot T
\, (\ln T)^2$ and from above by $C \cdot T \, (\ln T)^2$ for some
positive constants $c$ and $C$, see \cite{BS94} and
\cite{Gro96}. In fact the situation is quite rigid in dimension
$2$. Even the finiteness of the systolic volume over the set of
finite simplicial complexes of dimension $2$ holds. More
precisely, recall that
\begin{itemize}
\item the {\it systolic area} of a finitely presentable group $G$ is defined as
$$
\Sys(G) = \mathop{\inf}\limits_{\pi_1(P)=G} \Sys(P),
$$
where the infimum is taken over all finite simplicial complexes
$P$ of dimension $2$ with fundamental group $G$ ;
\item a finitely presentable group $G$ is said {\it of zero free index} if $G$ can not be written as a free product $H \ast F_n$ for some $n > 0$, compare \cite{RS08}.
\end{itemize}
Then the number of finitely presentable groups  $G$ of zero  free index such that $\Sys(G) \leq T$ does not exceed  $K^{T^3}$ where $K>1$ is an explicit constant, see \cite{RS08}.\\

The situation is rather different in higher dimensions. Let $M$ be
an essential manifold of dimension $m \geq 4$ whose fundamental
group is of zero  free index and $N$ a non-essential
manifold of the same dimension with  fundamental group of zero
 free index. The fundamental group $\pi_1(M) \ast \pi_1(N)$
of the connected sum $M \# N$ is still of zero  free index.
By Proposition \ref{prop:sommeconnexe1} we get
$$
\Sys(M \# N) \leq \Sys(M).
$$
(Recall that non-essential orientable manifolds have systolic volume equal to zero, see \cite{Bab92} for instance).
That is, while staying in the class of groups of zero  free
index we can considerably modify the fundamental group of a
manifold without increasing the systolic volume. So there is no
hope to obtain finiteness results in this context. This is why we
introduce the following.

\begin{definition}
Let $G$ be a finitely presentable group. A class $\a \in
H_m(G,\Z)$ is said to be reducible if there exists a proper subgroup $H
< G$ and a class $\b \in H_m(H, \Z)$ such that $i_{\ast}(\b)
= \a$ where $i$ denotes the canonical inclusion. In the contrary
case, the class $\a$ will be said irreducible.

Furthermore, we will say that a manifold $M$ is reducible (resp.
irreducible) if the image of its fundamental class $[M]$ (under
the classifying map) in $H_m(\pi_1(M), \Z)$ is a reducible (resp.
irreducible) class.
\end{definition}

\begin{exemple} Let $M$ be an aspherical manifold of dimension $m$ (that is $\pi_k(M) = 0$ for $k>1$). Then $M$ is irreducible.

On the other hand if $M$ and $N$ are non simply connected manifolds and $N$ is not essential then the connected sum $M \# N$ is 
always reducible.
\end{exemple}

\begin{exemple}
Let $G$ be a finite group and  $\a \in H_m(G, \Z)$ be a class of
order $|G|$. Then $\a$ is irreducible.
\end{exemple}

This last example shows that the fundamental class of a lens
manifold is irreducible.

Remark that it is possible that
\begin{itemize}
\item any multiple of an irreducible classe is irreducible as in the case $G= \Z_p$ for $p$ a prime number,
\item each multiple of an irreducible classe is reducible as in the case of tori $\T^m$.
\end{itemize}

On the other hand, there exist groups $G$ and classes $\a \in H_m(G,\Z)$ which are  {\it completely reducible} in the following sense: $\a$ is reducible, and any class $\b \in H_m(H, \Z)$ where $H$ is a proper subgroup $H \subset G$ and such that $i_{\ast}(\b) = \a$ is also reducible.\\

Given a positive number $T$ and a positive integer $m$, we
denote by $\F(m,T)$ the set of finitely presentable groups $G$
such that there exists an irreducible class  $\a \in H_m(G, \Z)$
with $\Sys(\a) \leq T$. Theorem \ref{th:nonfinitudeintro} of the introduction can be now restated as follows.

\begin{theorem}
The set $\F(m,1)$ is infinite for any dimension $m\geq 3$.
\end{theorem}

\begin{proof}
Let $p$ be a prime number and set
$$
G(p,m) := \underset{m}{\underbrace{\Z_p \oplus \ldots
\oplus\Z_p}}.
$$
Denote by  $\phi_p: \pi_1(\T^m) \longrightarrow G(p,m)$ the
natural projection and set
$$
\a(p,m) := (\phi_p)_{\ast}[\T^m] \in H_m(G(p,m), \Z).
$$

In order to prove that $\a(p,m) \neq 0$ in $H_m(G(p,m), \Z)$, we
will show that the reduction modulo $p$ of $\a(p,m)$ is not null
in $H_m(G(p,m), \Z_p)$. Consider the generators $v_1, \ldots, v_m$
of $H^1(G(p, m), \Z_p)$ corresponding to the natural projections
of $G(p, m)$ on each factor. The elements $u_i :=
(\phi_p)^{\ast}(v_i), \ \ i = 1, 2, ..., m$ generate the group
$H^1(\T^m, \Z_p)$, and $u := u_1 \cup  \ldots \cup u_m$ generates
the group $H^m(\T^m, \Z_p)$. So $u \cap [\T^m]_p = 1$ where
$[\T^m]_p$ denotes the reduction modulo $p$ of the fundamental
class $[\T^m]$. This implies
\begin{eqnarray}
\nonumber (v_1 \cup \ldots\cup v_m) \cap (\phi_p)_\ast[\T^m]_p & =
&
(\phi_p)^{\ast}(v_1) \cup \ldots \cup (\phi_p)^{\ast}(v_m) \cap [\T^m]_p\\
\nonumber & =& 1.
\end{eqnarray}
This proves the non-triviality of $(\phi_p)_{\ast}[\T^m]_p$, and thus of $(\phi_p)_{\ast}[\T^m]$.\\

We now prove the irreducibility of $\a(p,m)$. Let suppose the
contrary. Any proper subgroup $H$ of $G(p, m)$ is also a
$\Z_p$-vector subspace  of  $G(p, m)$ of dimension $k < m$.
Associated to some complement of  $H$ in $G(p, m)$, we
construct a projection map
$$
\pi: G(p, m) \longrightarrow H
$$
which is the identity on $H$. Fix a basis of the  free $\Z$-module
$\pi_1(\T^m)$ such that the composition
$$
\pi \circ \phi_p: \pi_1(\T^m) \longrightarrow H
$$
decomposes as
$$
\pi \circ \phi_p = \psi \circ \rho_\sharp,
$$
where $\rho_\sharp$ is induced by some projection $\rho : \T^m
\longrightarrow \T^k$ and $\psi : \pi_1(\T^k) \longrightarrow H$
corresponds to the reduction modulo $p$. Now assume that $\a(p,m)
= i_{\ast}(\b)$, where $\b \in H_m(H, \Z)$ and $i : H
\longrightarrow G(p, m)$ denotes the inclusion. Then
$$
\b = \pi_{\ast}(\a(p,m)) = \pi_{\ast}\circ (\phi_p)_{\ast}[\T^m] =
\psi_\ast \circ \rho_{\ast}[\T^m] = \psi_\ast(0) = 0
$$
as $\b = \pi_{\ast} \circ i_{\ast}(\b)$.
This gives a contradiction the class $\a(p,m)$ being non trivial.\\

In order to conclude the proof, remark that
$$
\Sys(\a(p,m)) = \Sys_{\phi_p}(\T^m) \leq \Sys(\T^m) \leq 1,
$$
for any $m \geq 3$ and any prime $p$.
\end{proof}

This theorem implies the following unexpected result in dimensions
$m \geq 4$ .

\begin{corollary}
For any dimension $m \geq 4$, there exists an infinite number of
irreducible orientable manifolds $M$ of dimension $m$ with
pairewise non-isomorphic fundamental groups such that $\Sys(M)
\leq 1$.
\end{corollary}

\begin{proof}
By construction, every class $\a(p,m)$ is representable by a
manifold. If $m \geq 4$, such a manifold can be modified by surgery
in order to get a new manifold denoted by $M(p, m)$ such that
$\pi_1(M(p, m)) = G(p, m)$ and $\Phi_{\ast}[M(p, m)] = \a(p, m)$,
where $\Phi: M(p, m) \longrightarrow K(G(p, m),1)$ denotes the
classifying map, see \cite{Bab06}. The infinite sequence of
manifolds $\{M(p, m)\}$ where $p$ runs over all prime numbers
gives an infinite sequence of irreducible orientable manifolds $M$
of dimension $m$ with pairewise non-isomorphic fundamental groups
such that $\Sys(M) \leq 1$.
\end{proof}

%\begin{remark}
%Each homology class  $\a(p, m)$ is of finite order $p$. A relatively more sophisticated construction furnishes an infinite sequence of irreducible homology classes of infinite order and whose systolic volume does not exceed a certain explicit constant $C > 0$.
%\end{remark}

The following natural question remains open :

\begin{question} Consider the systolic volume $\Sys(\cdot)$
as a function over all irreducible orientable manifolds of dimension $m$.
Does there exist a positive number $C$ such that the number
of distinct values of the function $\Sys(\cdot)$ less than $C$ is infinite ?
\end{question}

\section{Systolic volume of multiple classes}\label{sec:multiple}

Given a finitely presentable group $G$ and a homology class $\a \in H_\ast(G,\Z)$, how does the function  $\Sys(G,k\a)$ behave  in
term of the integer variable $k$ ? In this section, we prove Theorem \ref{th:multipleintro} which we restate here for the reader's convenience.
\begin{theorem}\label{th:multiple}
Let  $G$ be a finitely presentable group and  $\a \in H_m(G, \Z)$
where $m\geq 3$. Then there exists a positive number $C(G,\a)$
depending only on the pair $(G,\a)$ such that
\begin{equation}\label{eq:bornesup}
\Sys(G,k\a) \leq C(G,\a) \cdot \frac{k}{\ln(1 +k)}
\end{equation}
for any positive integer $k$.
\end{theorem}

Before proving this result, we put the question in a more general
context.

\subsection{Systolic volume of the sum of homology classes} Let $G_1$ and $G_2$ be
two finitely presentable groups and denote by
$i_j : G_j \hookrightarrow G_1 \ast G_2$ for $j=1,2$ the natural monomorphisms.
Fix  two integer homology classes  $\a_1\in H_m(G_1,\Z)$ and $\a_2\in
H_m(G_2,\Z)$. The natural isomorphism
$$
H_m(G_1\ast G_2, \Z) \simeq H_m(G_1, \Z) \oplus H_m(G_2, \Z), \ \
$$
for positive $m$ allows us to denote by $\a_1 + \a_2$ the class $(i_1)_{\ast}(\a_1) +
(i_2)_{\ast}(\a_2)$.

\begin{proposition}\label{prop:modif2}
If $m\geq 3$,
$$
\Sys(G_1\ast G_2,\a_1 + \a_2)\leq \Sys(G_1,\a_1)+\Sys(G_2,\a_2).
$$
\end{proposition}

\begin{proof}
For any $\varepsilon > 0$ and for $j=1,2$ we choose a geometric cycle
$(X_j,f_j)$ of dimension $m$ representing  $\a_j$ and satisfying
$$
\Sys_{f_j}(X_j) \leq \Sys(G_j, \a_j) + {\varepsilon \over 2}.
$$
The geometric cycle $(X_1\# X_2,f_1 \# f_2)$ obtained as the
connected sum of $(X_1,f_1)$ and $(X_2,f_2)$ represents the class
$\a_1+\a_2$. By the comparison principle (see \cite{Bab06}),
$$
\Sys_{f_1 \# f_2}(X_1 \# X_2) \leq \Sys_{f_1 \vee f_2}(X_1 \vee
X_2)=\Sys_{f_1}(X_1)+\Sys_{f_2}(X_2)
$$
where $(X_1\vee X_2,f_1 \vee f_2)$ denotes the wedge sum of
$(X_1,f_1)$ and $(X_2,f_2)$. As $\varepsilon$ can be chosen
arbitrarily small, we get the result.
\end{proof}

If $\a_1$ and $\a_2$ are two homology classes of dimension $ m$ of
the same group $G$, we deduce the following subadditivity property
of the systolic volume.

\begin{corollary}\label{cor:semiadditive}
Let $\a_1$ and $\a_2$ be two classes of $H_m(G, \Z)$ with $m\geq 3$.
Then
$$
\Sys(G, \a_1 + \a_2)\leq \Sys(G, \a_1)+ \Sys(G, \a_2).
$$
\end{corollary}

\begin{proof}
Indeed, if we denote by $\pi : G \ast G \to G$ the epimorphism
defined by $\pi \circ i_j=id_G$, then $\pi_\ast(\a_1 + \a_2)= \a_1
+ \a_2$. By Proposition \ref{prop:modif1}, we get the result.
\end{proof}

As a direct consequence of this corollary, we derive that
$$
\Sys(G,k\a)\leq k \cdot \Sys(G,\a)
$$
for any $\a \in H_m(G, \Z)$ with $m\geq 3$ and any integer  $k$.
In particular the limit
$$
\lim_{k \to \infty} \frac{\Sys(G,k\a)}{k}
$$
exists. Theorem \ref{th:multiple} permits us to conclude that this limit is always zero.

\subsection{Sublinear upper bound for the systolic volume of multiple classes}
Theorem \ref{th:multiple} is related to the behaviour of systolic
volume for connected sum and is a direct consequence
of the following result.

\begin{theorem}\label{th:sommeconnexe}
Let $X$ be a connected pseudomanifold of dimension $m\geq 3$ and denote by $\#_k X$ the connected sum of $k$ copies of $X$. There exists a positive number $C(X)$ depending only on the topology of
$X$ such that
\begin{equation}\label{eq:bornesup2}
\Sys(\#_k X) \leq C(X) \cdot \frac{k}{\ln (1 + k)}
\end{equation}
for any positive integer $k$.   If $k$ is large enough, this last
inequality (\ref{eq:bornesup2}) is satisfied for
$$
C(X) = m \cdot c(X) \cdot \ln c(X)
$$
where $c(X)$ stands for the minimal number of $m$-cubes in a cubical
decomposition of $X$.
\end{theorem}

Recall that a cubical decomposition of $X$ is a family of $m$-dimensional embedded topological cubes covering $X$ and such that any non-empty intersection between two cubes occurs along a lower dimensional face. In particular, any $(m-1)$-face belongs to exactly two $m$-cubes. The existence of such decompositions can be proved by considering a triangulation of $X$  and decomposing each $m$-simplex of this triangulation into $m+1$ $m$-cubes.

\begin{remark}\label{rem:cyclic}
Theorem \ref{th:sommeconnexe} substantially improves \cite[Theorem A]{BB05}, and the upper bound (\ref{eq:bornesup2}) still holds for a sequence of cyclic covering, the details being similar to those considered in the sequel (compare with \cite{BB05}). \\
\end{remark}

\begin{proof}
The proof relies on ideas and techniques appearing in \cite{BB05}.  
We consider a minimal cubical decomposition $\Te$ of $X$ with $c(X)$ elements. From each cube ${\mathcal C} \in \Te$ we remove an open cube  ${\mathcal C}'\simeq ]1/4,3/4[^m \subset [0,1]^m\simeq {\mathcal C}$ and denote by $X'$ the polyhedron thus obtained. Each set (called a {\it sleeve} of $X'$)
$$
\overline{{\mathcal C}} = {\mathcal C} \setminus
{\mathcal C}' \simeq \partial{\mathcal C} \times [0,1/2] ,
$$
is endowed with the product metric $g_{\partial{\mathcal C}} \times {1\over N} \, dt$ where $g_{\partial{\mathcal C}}$ denotes the metric on $\partial{\mathcal C}$ induced by a realization of ${\mathcal C}\simeq[0,1]^m \subset \R^m$ in the standard Euclidean vector space and $N$ is a positive integer to be chosen later.
Choose the metrics $\{g_{\partial{\mathcal C}}\}_{{\mathcal C} \in \Theta}$ to coincide on each non-empty intersection between two $m$-cubes, and denote by $g'$ the Riemannian polyhedral metric thus obtained on $X'$. 
By construction,

\begin{itemize}
\item the complex $X'$ is homeomorphic to  $X$ minus $c(X)$
disjoint open $m$-disks, and in particular
$\pi_1(X') = \pi_1(X)$ ;

\item $\vol(X', g') = m \cdot {1 \over N}  \cdot c(X)$ ;

\item the obvious retraction of  $X'$ onto $\Te^{(m-1)}$ (where $\Te^{m-1}_i$ denotes the $(m-1)$-skeleton of $\Te$) is distance non-increasing.

\end{itemize}

Fix an $m$-cube ${\mathcal C} \in \Te$. We define the star $\overline{\mbox{st}}({\mathcal C})$ as the union of all cubes with non-empty intersection with ${\mathcal C}$ (including ${\mathcal C}$ itself). The open star $\mbox{st}({\mathcal C})$ is then defined as the union of the interiors of all cubes having a non-empty intersection with ${\mathcal C}$, so the closure of $\mbox{st}({\mathcal C})$ is exactly $\overline{\mbox{st}}({\mathcal C})$. The following lemma will be important in what follows. 

\begin{lemma}\label{lem:inf}
If a relative curve $\gamma$ of $(X',\overline{{\mathcal C}})$ is not entirely included in $X'\cap \mbox{st}({\mathcal C})$, then  $l_{g'}(\gamma) \geq 2$.
\end{lemma}

\begin{proof}[Proof of the Lemma]
Let $p$ belongs to the boundary $\partial \, (\overline{\mbox{st}}({\mathcal C}))=\overline{\mbox{st}}({\mathcal C})\setminus \mbox{st}({\mathcal C}) $. Then 
$$
d_{g'}(p,\partial {\mathcal C})\geq 1
$$
which proves the Lemma. 

Indeed, suppose that this is not the case. The point $p$ lies in the $(m-1)$-face $F$ of an $m$-cube ${\mathcal C}_1 \subset \overline{\mbox{st}}({\mathcal C})$. But $F$ belongs to exactly two $m$-cubes ${\mathcal C}_1$ and ${\mathcal C}_2$. As $d_{g'}(F,\partial {\mathcal C})\leq d_{g'}(p,\partial {\mathcal C})<1$, we conclude that $F\cap {\mathcal C}\neq \emptyset$ which implies that  ${\mathcal C}_2 \subset \overline{\mbox{st}}({\mathcal C})$. This together with the condition $d_{g'}(p,\partial {\mathcal C})<1$ implies that $p\in \mbox{st}({\mathcal C})$ which is a contradiction.
\end{proof}

According to the construction of graphs with large girth by \cite{ES63}, there exists for any positive integers $n$ and $N$ such that
$$
n\geq 2\sum_{t=1}^{N-2}(c(X)-1)^t
$$ 
a $c(X)$-regular graph $\Gamma$ with $2n$ vertices whose girth---defined as the least number of edges composing a cycle---is at least $N$. The graph $\Gamma$ can be thought as a $1$-dimensional polyhedron on which we define the PL-metric $h$ for which all edges have length ${1\over N}$. The girth's bound thus says that $\sys(\Gamma,h)\geq 1$.

Denote by $v_1,\ldots,v_{2n}$ the vertices of $\Gamma$ and let $(X'_1,g'_1),\ldots,(X'_{2n},g'_{2n})$ be $2n$ copies of the Riemannian polyhedron $(X',g')$. We define a map $F : X'_1 \cup \ldots \cup X'_{2n} \to \Gamma$ as follows. For each $i=1,\ldots,2n$,
\begin{itemize}
\item  Set $F(p)=v_i$ if $p \in \Te^{m-1}_i$ where $\Te^{m-1}_i$ denotes the $(m-1)$-skeleton of the cubical decomposition $\Te_i$ corresponding to $X'_i$. 
\item Denote by $e^i_1,\ldots,e^i_{c(X)}$ the edges adjacent to $v_i$, and by $\overline{{\mathcal C}}^i_1,\ldots,\overline{{\mathcal C}}^i_{c(X)}$ the sleeves of $X'_i$ corresponding to the cubical decomposition $\Theta_i$.  For each $j=1,\ldots,c(X)$, each edge $e^i_j$ is isometric to $([0,1],dt)$ with $0$ being identified with the vertex $v_i$. Then for each point $p=(x,t)$ belonging to the sleeve $\overline{{\mathcal C}}^i_j \simeq \partial{\mathcal C}^i_j \times [0,1/2]$ set $F_i(p)=t$.
\end{itemize}
 Observe that each restriction $F_{\vert X'_i} : (X'_i,g'_i) \to (\Gamma,h)$ is distance non-increasing.

Now to each edge $e=(v_i,v_j) \in E(\Gamma)$ correpond exactly two sleeves $\overline{{\mathcal C}}$ and $\overline{{\mathcal C}'}$ belonging respectively to the two distincts copies $X'_i$ and $X'_j$ of $X'$ and such that their images under $F$ cover the edge $e$. We choose a PL-isometry denoted by $\phi_e$ between $(\partial{\mathcal C} \times \{1/2\},g'_i)$ and  $(\partial{\mathcal C'} \times \{1/2\},g'_j)$. The quotient space
$$
X(n)=(X'_1\cup \ldots \cup X'_{2n})/\{\phi_e\}_{e \in  E(\Gamma)}
$$ 
is a polyhedron homeomorphic to
$$
(\mathop{\#}_{2n}X)\#(\mathop{\#}_{n(c-2)+1}S^1 \times S^{m-1}).
$$
The metrics $\{g'_i\}_{i=1,\ldots,n}$ are compatible with the quotient map $X'_1 \cup \ldots \cup X'_n \twoheadrightarrow X(n)$ and give rise to a Riemannian polyhedral metric on $X(n)$ that we denote by $g(n)$.
The map $F$ induces a new map still denoted by $F$ from $X(n)$ onto $\Gamma$. By construction this map $F : (X(n),g(n)) \to (\Gamma,h)$ is distance non-increasing.

\medskip

According to \cite[Proposition 1]{BB05} the
addition of $1$-handles does not change the value of the systolic
volume, and therefore  
$$
\Sys(\#_{2n}X) = \Sys(X(n))\leq \frac{\vol(X(n), g(n))}{(\sys(X(n),g(n))^m}. 
$$
We have 
$$
\vol(X(n),g(n))=2n \cdot \vol(X',g')= 2n \cdot m \cdot {1Ê\over N} \cdot c(X). 
$$

\begin{lemma}
\begin{equation}\label{eq:sys}
\sys(X(n), g(n)) \geq 1.
\end{equation}
\end{lemma}

\begin{proof}
Let $\gamma : [0,1] \to X(n)$ be a closed curve which is not contractible.

\smallskip

If its image $F(\gamma)$ is not contractible in $\Gamma$, then
$$
l_{g(n)}(\gamma) \geq l_h(F(\gamma)) \geq \sys(\Gamma, h) \geq 1
$$
as $F$ is distance non-increasing.

\medskip

Now suppose that $F(\gamma)$ is contractible in $\Gamma$. We will show that
$l_{g(n)}(\gamma) \geq 2$ using Lemma \ref{lem:inf}.

First of all, we can assume that $\gamma$ is minimizing in its own
homotopy class. The contractibility of $F(\gamma)$ implies that we can find an
edge $[v_i,v_j]$ of $\Gamma$ joining two vertices $v_i$ and $v_j$,
a point $v \in ]v_i,v_j]$ and a triplet $0 \leq t_1 < t_{\ast} < t_2 \leq 1$ such that 
\begin{itemize}
\item $v = F(\gamma(t_{\ast}))$ ;
\item $F(\gamma([t_1, t_2]) \subset [v_i, v]$ ;
\item $F(\gamma(t_1))=F(\gamma(t_2))=v_i$.
\end{itemize}
Denote by $m_{ij}$ the midpoint of the edge $[v_i,v_j]$, and by
$$
\overline{\mathcal C}_i=\overline{F^{-1}(]v_i,m_{ij}[)} \, \, \, \, \text{and} \, \, \, \, \overline{\mathcal C}_j=\overline{F^{-1}(]m_{ij},v_j[)}
$$
the two corresponding sleeves.

\smallskip

If $v \in ]v_i,m_{ij}]$, we contract the portion of the curve
$\{\gamma(t)\}_{t \in [t_1, t_2]}$ using the projection $\pi_i : \overline{\mathcal C}_i\simeq \partial {\mathcal C}_i\times [0,1/2]  \to \partial {\mathcal C}_i\times \{0\}$ given by the formula $\pi_i(p,t)=(p,0)$. 
This contraction strictly decreases the length of $\gamma$ which is in
contradiction with its minimality. 
If $v \in ]m_{ij},v_j[$, we contract the portion of the curve
$\{\gamma(t)\}_{t \in [t_1, t_2]}$ lying in $\overline{\mathcal C}_j$ using the projection $\pi_j : \overline{\mathcal C}_j\simeq \partial {\mathcal C}_j\times [0,1/2]  \to \partial {\mathcal C}_j\times \{1/2\}$  given by the formula $\pi_j(p,t)=(p,1/2)$. 
This contraction also strictly decreases the length of $\gamma$ which is again in
contradiction with its minimality. 
So the point $v$ must coincides with $v_j$.

\smallskip

The restriction of $\gamma$ to the interval $[t_1, t_2]$ is thus
the concatenation  $\gamma_1 \star \gamma_2 \star \gamma_3$ of
three curves with $\gamma_{1}$,  $\gamma_ 3 \subset
F^{-1}([v_i,v_j[)$ and $\gamma_2 \subset
F^{-1}(v_j)=\Te_j^{m-1}$. We will show that $\gamma_2$ is not entirely contained in the open star $\mbox{st}({\mathcal C}_j)$ corresponding to the sleeve $\overline{\mathcal C}_j$, and thus
$$
l_{g(n)}(\gamma) \geq l_{g'_j}(\gamma_2) \geq 2
$$
according to Lemma \ref{lem:inf} which concludes the proof of the lemma.

\smallskip

Indeed suppose that $\gamma_2$ is entirely contained in the open star $\mbox{st}({\mathcal C}_j)$. 
Observe first that $\gamma_2 \not\subset \overline{\mathcal C}_j$. Otherwise we contract the portion of the curve $\{\gamma(t)\}_{t \in [t_1, t_2]}$ lying in $\overline{\mathcal C}_j$ using the projection $\pi_j : \overline{\mathcal C}_j\simeq \partial {\mathcal C}_j\times [0,1/2]  \to \partial {\mathcal C}_j\times \{1/2\}$  given by the formula $\pi_j(p,t)=(p,1/2)$. This contraction strictly decreases the length of $\gamma$ which is in contradiction with its minimality. Now the orthogonal projections of  $\mbox{st}(\overline{\mathcal C}_j)\cap \Te_j^{m-1}$ on each face of $\overline{\mathcal C}_j \cap \Te_j^{m-1}$ are correctly defined and coherents. So by projecting $\gamma_2$ orthogonally on $\overline{\mathcal C}_j \cap \Te_j^{m-1}$ we do not change the homotopy class of $\gamma$, and as $\gamma_2 \not\subset \overline{\mathcal C}_j$, the length of $\gamma$ strictly decreases. This is again a contradiction with its minimality.
\end{proof}

So
\begin{equation}\label{eq:sup}
\Sys(\#_{2n}X) \leq \frac{\vol(X(n), g(n))}{(\sys(X(n),
g(n))^m} \leq m\cdot {2n\over N}\cdot c(X)
\end{equation}
for any positive integers $n$ and $N$ such that 
$$
n \geq {2 \over c(X)-2}[(c(X)-1)^{N-1} - (c(X)-1)].
$$

\medskip

Fix a positive integer $n$. As $c(X) \geq 2m + 1 \geq 7$, we choose $N$ such that
$$
n\in [(c(X)-1)^{N-1},(c(X)-1)^N],
$$
and get that
$$
\Sys(\#_{2n}X) \leq m \cdot c(X) \cdot \ln (c(X)-1)\cdot {2n \over \ln (2n+1)}
$$
according to (\ref{eq:sup}).
Furthermore
\begin{eqnarray}
\nonumber \Sys(\#_{2n+1}X) & \leq & \Sys(\#_{2n}X) + \Sys(X)\\
\nonumber & \leq & m \cdot c(X) \cdot \ln (c(X)-1)\cdot {2n \over \ln (2n+1)} + \Sys(X)\\
\nonumber & \leq & m \cdot c(X) \cdot \ln (c(X)-1)\cdot {2n+1 \over \ln (2n+2)}+ \Sys(X). \nonumber
\end{eqnarray}
For $n$ large enough, we thus get the universal upper bound
(\ref{eq:bornesup2}) which concludes the proof.
\end{proof}

\subsection{Homology classes with positive simplicial volume}\label{sub:hyp}

Recall the following definition (see \cite{Gro82}).

\begin{definition}
Let $X$ be a pseudomanifold of dimension $m$. Its {\it simplicial
volume} is the quantity
 $$
 \|X\|_{\Delta} = \inf \, \{\sum_i |r_i| \, | \, [X] = \sum_i r_i \sigma_i^m  \},
$$
where the infimum is taken over the set of representations of the
fundamental class $[X]$ by singular simplicial chains with real
coefficients.

If $G$ denotes a finitely presentable group and $\a$ a homology
class of dimension $m$, the simplicial volume of $\a$ is then the
number
$$
\|\a\|_{\Delta}=\inf \, \{\|X\| \mid X \, \text{representing} \ \
\a \}.
$$
\end{definition}
For homology classes whose simplicial volume is positive, the
function $\Sys(G,k\a)$ goes to infinity and the following result provides a better information about its asymptotic behaviour.

\begin{corollary}\label{cor:asymptotique} Let $G$ be a finitely presentable
group and $\a \in H_m(G, \Z)$ be a homology class of dimension $m \geq 3$ such that
$\|\a\|_{\Delta} > 0$. Then there exists two positive numbers
$C(G,\a)$ and
 $C'(G,\a)$ depending only on the pair $(G,\a)$ such that
$$
C'(G,\a) \cdot \frac{k}{(\ln(1 +k))^m} \leq \Sys(G,k\a)
\leq C(G,\a) \cdot \frac{k}{\ln(1 + k)}
$$
for any positive integer $k$.
\end{corollary}

\begin{proof}
The lower bound is a direct consequence of the following
inequality of Gromov (see \cite[Theorem 6.4.D']{Gro83}): any
pseudomanifold $X$ of dimension $m$ satisfies the inequality
$$
C_m'\frac{\|X\|_{\Delta}}{(\ln(2 + \|X\|_{\Delta}))^m} \leq \Sys(X),
$$
where $C'_m$ is a positive number depending only on the dimension
$m$. It remains to remark that, if $X$ represents the class $ka$,
then $\|X\|_{\Delta} = \|k\a\|_{\Delta} = k\|\a\|_{\Delta}$. The
upper bound then follows by Theorem \ref{th:multiple}.
\end{proof}

\subsection{Large oscillations of systolic volume}

The following example shows that the function $k
\mapsto \Sys(G,k\a)$ may have arbitrarily large oscillations. \\

Let $m = 2l + 1 \geq 3$ be an odd integer and $q\geq 2$ an
integer. Let $X$ be an essential manifold of dimension $m$ (for
example aspherical) and $f:X \to K(\pi_1(X), 1)$ be its
classifying map. If $X$ is not aspherical, we assume that the
image of its fundamental class $f_{\ast}[X]$ is an element of
infinite order in $H_m(\pi_1(X), \Z)$. Set $a=f_{\ast}[X]$. Fix a
generator $\b \in H_m(\Z_q, \Z) = \Z_q$. For each $1 \leq l \leq
q-1$, we choose a normal representation of $l\b$ by a manifold
$Y_l$ with  $\pi_1(Y_l) = \Z_q$. For relatively prime $l$ and $q$, the
corresponding lens space can be chosen to be $Y_l$. Set
$$
D = \mathop{\max}\limits_{1 \leq k \leq q} \Sys(\pi_1(X),k\a)
$$
and fix any positive number $C$. Consider the free product
$$
G_n = \pi_1(X) \ast \underset{n}{\underbrace{\Z_q \ast \ldots \ast
\Z_q}},
$$
and pick in
$$
H_m(G_n, \Z) = H_m(\pi_1(X), \Z) \oplus H_m(\Z_q, \Z) \oplus
H_m(\Z_q, \Z) \oplus ... \oplus H_m(\Z_q, \Z)
$$
the class ${\bf c} = \a \oplus \underset{n}{\underbrace{\b \oplus
\ldots \oplus \b}}$. If $X_k$ is a manifold representing the class
$k\a$, then
$$
X_k \# \underset{n}{\underbrace{Y_k \# \ldots \# Y_k}}
$$
represents the class $k{\bf c}$. By Corollary \ref{cor:somme},
$$
\Sys(G_n,k{\bf c}) \geq \Sys(\underset{n}{\underbrace{Y_k \#
\ldots \# Y_k}}).
$$
 Now Theorem {\bf A} of \cite{Sab07} implies that if  $n$ is chosen large enough, we have $\Sys(G_n,k{\bf c}) > C$ for any $1
\leq k \leq q-1$. Besides $q{\bf c} = q\a$ in $H_m(G_n, \Z)$, and so $\Sys(G_n,q{\bf c}) \leq D$.\\

\subsection{Systolic generating function}

If $G$ is a finitely presentable group and $\a \in H_m(G, \Z)$ is
a homology class of infinite order, the study of the sequence
$\{\Sys(G,k\a)\}_{k = 1}^{\infty}$ is equivalent to the study of
analytic properties of the following generating function:
$$
\sigma_{G, \a}(z) = \mathop{\sum}\limits_{k =
1}^{\infty}\Sys(G,k\a) \cdot z^k .
$$
If the group  $G$ is clearly identified by the context, we
simplify the notation into $\sigma_\a(z)$. The upper bound
(\ref{eq:bornesup}) implies that $\sigma_\a(z)$ is an analytic
function on the disk $|z| < 1$. Furthermore the complex point $z =
1$ is a singular point of this function as  $\Sys(G,k\a) \geq
\sigma_m > 0$.

In general there is no hope for $\sigma_{G, \a}(z)$ to be a
rational function. Indeed, if  $\a$ is a class with positive
simplicial volume, Corollary \ref{cor:asymptotique} teachs us that
$z = 1$ is not a pole of the corresponding systolic generating
function: so $\sigma_\a(z)$ is not rational. It is well known that
the rationality of the generating function of a numerical sequence
is equivalent to the recurrence of this sequence. We deduce the
following

\begin{proposition}\label{prop:generatrice}
Let $G$ be a finitely presentable group and $\a \in H_m(G, \Z)$ a
homology class of infinite order. If the simplicial volume of $\a$
is positive, the sequence of systolic volumes $\{\Sys(G,k\a)\}_{k
= 1}^{\infty}$ does not satisfy any linear recurrent equation.
\end{proposition}

Nevertheless the rationality of  $\sigma_\a(z)$ seems plausible
for  classes $\a$ with a bounded  sequence of systolic volume
$\{\Sys(G,k\a)\}_{k = 1}^{\infty}$. Tori give a model of this type
of behaviour for multiple classes.

\begin{proposition}
Let $3\leq m\leq n$ be two integers. Any class $\a \in H_m(\Z^n,
\Z)$ satisfies the inequality
$$
\Sys(\Z^n,\a) \leq {n \choose m} \cdot \Sys(\T^m),
$$
where ${n \choose m}$ denotes the binomial coefficient.
\end{proposition}

\begin{proof} Fix a basis of $H_m(\Z^n, \Z)$ composed of
embedded $m$-tori, and write the class $\a$ in this basis:
$$
\a=\sum_{i=1}^{{n \choose m}} k_i [\T^m_i]
$$
where $k_i \in \Z$ for each $i$. By Corollary
\ref{cor:semiadditive},
\begin{equation} \label{eq:subad}
\Sys(G,\a)\leq \sum_{i=1}^{n \choose m} \Sys(G,k_i [\T^m_i]).
\end{equation}
Observe that
\begin{equation}\label{eq:subad2}
\Sys(\Z^n,k[\T_i^m]) \leq  \Sys(\T^m)
\end{equation}
for any integer $k$. In fact, if $f : \T^m \to \T_i^m$ denotes a
map of degree $k$, the geometric cycle $(\T^m,f)$ represents the
class $k[\T_i^m]$. By adding $1$-handles, we can normalize this
representation into a geometric cycle $(\T^m \# (S^1\times
S^{m-1}) \# \ldots \# (S^1\times S^{m-1}),\tilde{f})$. We get
\begin{eqnarray*}
\Sys(\Z^n,k[\T_i^m])&=&\Sys_{\tilde{f}}(\T^m \# (S^1\times
S^{m-1}) \# \ldots \# (S^1\times
S^{m-1}))\\
&\leq& \Sys(\T^m \# (S^1\times S^{m-1}) \# \ldots \# (S^1\times S^{m-1}))=\Sys(\T^m).\\
\end{eqnarray*}
Now we deduce the result by combining inequalities
(\ref{eq:subad}) and (\ref{eq:subad2}).
\end{proof}

We close this chapter with the following

\begin{conjecture}
{\it If $\a=[\T^m] \in H_m(\Z^m, \Z)$, then the associated
systolic generating function is}
$$
\sigma_\a(z) = \Sys(\T^m) \cdot {z \over 1 - z }.
$$
\end{conjecture}

\section{The Heisenberg group, nilmanifolds and the Waring problem}\label{sec:nilp}

The nilpotent groups give particularily interesting examples: the
systolic volume of multiples of certain homology classes are
bounded, albeit certain of these multiples admit (non-normalized)
representations by manifolds whose systolic volume is not bounded.
This phenomena already appears in the simplest case of nilpotent
non abelian group, that is the Heisenberg group.

%More precisely, there exist a nilmanifold $M_{\H}$ of dimension 3 and a sequence $\{M_n\}_{n=1}^{\infty}$ of cyclic covering with $n$ sheets of $M_\H$ such that \begin{equation}\label{eq:comportement} (\ln n)^{1 - \varepsilon} \lesssim \Sys(M_n) \lesssim {\frac{n}{\ln n}}, \end{equation} fo any positive $\varepsilon$. All these manifolds have zero simplicial volume, and the asymptotic lower bound is obtained thanks to Theorem \ref{th:torsion}. The upper bound is a direct consequence of the version of Theorem  \ref{th:sommeconnexe} for cyclic covering.

\subsection{Nilmanifolds and the Waring problem}
%The idea to use the solution to the Waring problem in order to estimate from above the function $\Sys(G, k\a)$  perfectly applies to a pair $(G,\a)$ where  $G$ is nilpotent gardued group without torsion and  $\a$ denotes the findamentalm class of the corresponding nilmanifold. \\
Consider a nilpotent group $G$ of finite type without torsion. The
classical result of  Mal'cev \cite{Malc49} implies that there
exists a simply connected nilpotent Lie group  $\G(G)$ such that
$G$ embeds in  $\G(G)$ as a lattice, that is as a cocompact
discret subgroup. Denote by $ \L(G)$ the Lie algebra of $\G(G)$
and suppose that $\L(G)$ is graded in the following way:
\begin{equation}\label{eq:gradue}
 \L(G) = \mathop{\oplus}\limits_{k=1}^s \L_k , \ \ [\L_i, \L_j] \subset \L_{i+j},
\end{equation}
where $\L_{i+j} = 0$ if $i + j > s$. We do not suppose here that
\begin{equation}\label{eq:gradue1}
\{\L_{(i)} = \mathop{\oplus}\limits_{k=i}^s \L_k\}_{i=1}^s
\end{equation}
is a lower central series, $s$ being not in general the
nilpotency class of  $\L(G)$. For any $t \in \R$,
a natural homothety $\delta_t$  is associated to the decompostion
(\ref{eq:gradue}) by the formula:
$$
\delta_t(v) = t^k v \ \ \mbox{if} \ \ v \in \L_k.
$$
This homothety $\delta_t$ is an endomorphism of $\L(G)$ for any
real parameter $t$. By the Baker-Campbell-Hausdorff formula, and
taking into account the structure (\ref{eq:gradue}), the homothety
$\delta_t$ generates a homothety $\Delta_t$ of $\G$.

\begin{definition}
The nilpotent  group $G$ is said graded if  there exists a
graduation (\ref{eq:gradue}) of the corresponding Lie algebra
$\L(G)$ such that for integer parameters the corresponding
homotheties of $\G(G)$   preserve the lattice $G \subset \G(G)$,
that is
$$
\Delta_n(G) \subset G, \ \forall n \in \Z.
$$
\end{definition}

If $G$ is a nilpotent graded group, let
$$
d(G) = \mathop{\sum}\limits_{k=1}^sk\dim \L_k
$$
be the weighted dimension of its corresponding Lie algebra $\L(G)$. Remark that if the sequence of subalgebras (\ref{eq:gradue1}) is the lower central series, $d(G)$ coincides with the degree of polynomial growth of $G$, see \cite{Wolf68} and \cite{Bass72}.\\

Denote by $M = \G(G)/G$ the nilmanifold corresponding to the
nilpotent group $G$. If $G$ is graded, the homothety  $\Delta_n:
\G(G) \longrightarrow \G(G)$ defined using the graduation induces
for every positive integer $n$ a map
$$
\widetilde{\Delta}_n: M \longrightarrow M
$$
of degree $n^d$ with $d = d(G)$. Let $\a = [M] \in H_m(G, \Z)$ be
the fundamental class of $M$ and $k$ be a positive integer. We can
represent the class $k\a$ by the connected sum of a uniformily
bounded number of copies of $M$ as follows. By a result of Hilbert
(see \cite{Ellis71}), there exists an integer ${\mathcal K}(d)$
such that
\begin{equation}\label{eq:decompose}
k = \mathop{\sum}\limits_{i=1}^s a_i^d,
\end{equation}
where each coefficient $a_i$ is a positive integer and $s \leq
{\mathcal K}(d)$. Now the class $k\a$ is represented by the
geometric cycle $(\#_{i=1}^s M_i,f)$ where $M_i \simeq M$ for $i =
1, \dots, s$ and
$$
f: \mathop{\#}\limits_{i=1}^s M_i \longrightarrow
\mathop{\bigvee}_{i=1}^s M_i
\mathop{\longrightarrow}^{{\mathop{\bigvee}\limits_{i=1}^s}\widetilde{\Delta}_{a_i}}
M_{\H},
$$
the first map being the contraction of the connected sum into a
wedge. We easily compute that $\mbox{deg}f = \sum_{i=1}^s a_i^d=k$
and so
$$
\Sys(G,k\a) \leq \Sys(\#_{i=1}^s M).
$$
Then we apply Proposition \ref{prop:sommeconnexe1} in order to
derive the following result.

\begin{theorem}\label{th:nil}
Let $G$ be a graded nilpotent group. If $\a=[\G(G)/G]$ denotes the
fundamental class of the corresponding nilmanifold, then
$$
\Sys(G,k\a) \leq {\mathcal K}(d(G)) \cdot \Sys(G,\a)
$$
for any positive integer $k$.
\end{theorem}

\subsection{Family of lattices in the Heisenberg group}
Consider the Heisenberg group of dimension $3$ composed of the
following set of upper triangular matrices:
$$
\H= \left\{\left(
\begin{array}{ccc}
1 & x & z \\
0 & 1 & y \\
0 & 0 & 1
\end{array}
\right) , \, x, y, z \in \R \right\}.
$$
The subset $\H(\Z)$ of matrices of $\H$ with integer coefficients
({\it i.e.} for which $x, y, z \in \Z$) is a lattice, and we
denote by $M_{\H} = \H / \H(\Z)$ the corresponding nilmanifold.
The fundamental group $\H(\Z)$ of $M_{\H}$ satisfies the
assumptions of Theorem \ref{th:nil}. In fact the homotheties
$\{\Delta_t\}_{t>0}$ are given by the formula
$$\Delta_t
\left(
\begin{array}{ccc}
1 & x & z \\
0 & 1 & y \\
0 & 0 & 1
\end{array}
\right) = \left(
\begin{array}{ccc}
1 & tx & t^2z \\
0 & 1 & ty \\
0 & 0 & 1
\end{array}
\right),
$$
so $\Delta_n(\H(\Z)) \subset \H(\Z)$ for any integer $n\geq 1$.
The map $\Delta_n$ factorizes through a map
$$
\widetilde{\Delta}_n : M_{\H} \longrightarrow M_{\H}
$$
for which $\mbox{deg} (\widetilde{\Delta}_n) = n^4$. The
resolution of the Waring problem for the sum of fourth powers (see
\cite{BDD86}) gives that any integer number decomposes into a sum
of at most $19$ fourth powers. That is, with the notation of
Theorem \ref{th:nil}, we have $d(\H(\Z))=4$ and ${\mathcal K}(4) =
19$. Theorem \ref{th:Heisintro} of the introduction now easily follows.

\begin{corollary}\label{coro:nil1}
Let $\a = [M_{\H}] \in H_3(\H(\Z), \Z)$ be the fundamental class
of $M_\H$. Then
$$
\Sys(\H(\Z),k\a) \leq 19 \cdot \Sys(\H(\Z),\a)
$$
for any positive integer $k$.
\end{corollary}

The different lattices of $\H$ give rise to nilmanifolds whose
systolic behaviour is particularly interesting. Consider the
sequence of lattices $\{\H_n(\Z)\}_{n=1}^{\infty}$ of $\H$, where
$\H_n(\Z)$ is the subset of matrices of $\H$ such that  $x \in
n\Z$ and $y, z \in \Z$. Denote by $M_n = M_{\H_n} = \H / \H_n(\Z)$
the corresponding nilmanifolds. The manifold $M_n$ is a cyclic
covering of $M_{\H}$ with  $n$ sheets, so
\begin{equation}\label{eq:bornesup3}
\Sys(M_n) \leq C\frac{n}{\ln n},
\end{equation}
according to the version of Theorem \ref{th:sommeconnexe} for
cyclic coverings. The fact that the function $\Sys(M_n)$ goes to
infinity is not obvious. For instance the simplicial volume of
these manifolds is zero, and thus the corresponding lower bound
(see Corollary (\ref{cor:asymptotique})) does not apply. We are nevertheless able to prove the following proposition, see \cite{BBB14}.

\begin{proposition}\cite{BBB14}\label{prop:nil2}
The function $\Sys(M_{\H_n})$ satisfies the following inequality:
$$
\Sys(M_{\H_n}) \geq a\frac{\ln n}{\exp(b\sqrt{\ln \ln n})} ,
$$
where $a$ an $b$ are two positive constants. In particular,
$$
\lim_{n\to +\infty} \Sys(M_{\H_n})=+\infty.
$$
\end{proposition}

As the cover $M_n \longrightarrow M_{\H}$ has $n$ sheets, the
manifold $M_n$ represents the class $n[M_{\H}] \in H_3(\H(\Z),
\Z)$ for any positive integer $n$. This representation is not
normalized, and Corollary \ref{coro:nil1} together with
Proposition \ref{prop:nil2} shows that the assumption of
normalization can not be dropped in Proposition \ref{thm:realisation}.

\end{document}